\documentclass[11pt]{amsart}
\usepackage{preamble}
\usepackage{tikz-cd}

\begin{document}

\title[On the Effective Nonvanishing of Varieties of Kodaira Dimension Zero]{On the Effective Nonvanishing of Varieties of Kodaira Dimension Zero}


\begin{abstract}
Given a smooth projective variety $X$ of Kodaira dimension zero, we show that there exists a constant $m$ depending on two invariants of the general ﬁber of the Albanese map, such that $|mK_X|\neq\emptyset$.
\end{abstract}

\author{Yiming Zhu}
\address{Department of Mathematics, Southern University of Science and Technology, 1088 Xueyuan Rd, Shenzhen 518055, China} \email{zym18119675797@gmail.com}

\maketitle

\setcounter{tocdepth}{1}
\tableofcontents

We work over $\Cc$.

\section{Introduction} 
Let $X$ be a smooth projective variety of Kodaira dimension zero and $f:X\to Y$ its Albanese map. A fundamental result about $f$ due to Kawamata \cite[Theorem 1]{kawamata1981characterization} is that $f$ is an algebraic fiber space. We denote by $F$ the general fiber of $f$. A crucial fact about $F$, proved by Cao-Păun \cite{CapPaun2017Kodairadimension}, is that $F$ has Kodaira dimension zero. Let $b:=\min\{k\in\Nn\mid|kK_F|\neq\emptyset\}$, and $\Tilde{F}$ a resolution of the cover branched over the unique divisor in $|bK_F|$. We set $\beta_{\tilde{F}}:=\dim H^{\dim {\tilde{F}}}_{prim}({\tilde{F}},\Cc)$ as the primitive middle Betti number of $\tilde{F}$, and $N(\beta_{\tilde{F}}):=lcm\{k\in\Nn\mid\varphi(k)\leq\beta_{\tilde{F}}\}$, where $\varphi$ is Euler's totient function. The following is our main result.

\begin{theorem}\label{thm:main 1}
The sheaf $f_*\omega^{N(\beta_{\Tilde{F}})b}_X\cong O_Y$, and the linear system $|N(\beta_{\Tilde{F}})bK_X|$ is nonempty.
\end{theorem}

\begin{proof}
It follows from Theorem \ref{HPS}, Lemma \ref{lemma}, and Theorem \ref{thm:main 2}.    
\end{proof}

The following result due to Hacon-Popa-Schnell is crucial to our proof.

\begin{theorem}\cite[Corollary 4.3]{HaconPopaSchnell}\label{HPS}
The sheaf $f_*\omega^{mb}_X$ is a numerically trivial line bundle for all $m\geq0$.
\end{theorem}

\begin{example}
Let $X$ be a smooth projective variety of Kodaira dimension zero. Set $q(X):=h^1(O_X)$. If $q(X)=\dim X-1$, then $|12K_X|\neq\emptyset$. If $q(X)=\dim X-2$, then $|NK_X|\neq\emptyset$, where $N=lcm\{k\in\Nn\mid\varphi(k)\leq21\}$.
\end{example}

\begin{remark}
Note that the constant $m$ appearing in work on the effectivity of Iitaka fiber space by Birkar-Zhang \cite{birkar2016effectivity} also depends on $\beta_{\Tilde{F}}$ and $b$, where $F$ is the general fiber of the Iitaka fiber space, but in our setting, $F$ is the general fiber of the Albanese map.   
\end{remark}

\section{Parabolic fiber space
}

We say a fiber space $f:X\to Y$ is parabolic if its general fiber has Kodaira dimension zero. Given a parabolic fiber space, we shall denote by $F$ its general fiber, by $b$, $\tilde{F}$ and $N(\beta_{\tilde{F}})$ similarly as in the Introduction. We follow the notations of work on canonical bundle formula by Fujino and Mori \cite[Section 5]{Mori1985classification}, \cite{FujinoMori2000}.

\begin{lemma}\label{lemma}
Let $f:X\to Y$ be a parabolic fiber space. If the line bundle $(f_*\omega^{mb}_{X/Y})^{\vee}$ (where $^{\vee}$ denotes the reflexive hull) is numerically trivial for all $m\geq0$, then the sheaf $f_*\omega^b_{X/Y}$ is a line bundle, the $\Qq$-divisors $L_{X/Y}$, $L^{ss}_{X/Y}$ are integral and $O_Y(L_{X/Y})\cong O_Y(L^{ss}_{X/Y})\cong f_*\omega^b_{X/Y}$.
\begin{proof}
For all $m\geq0$, since $c_1(f_*\omega^{mb}_{X/Y})=0$, the canonical singular Hermitian metric on $f_*\omega^{mb}_{X/Y}$ is smooth and flat, and $f_*\omega^{mb}_{X/Y}$ is a line bundle \cite[4.6]{HaconPopaSchnell}. By considering the nonzero map $(f_*\omega^{b}_{X/Y})^m \to f_*\omega^{mb}_{X/Y}$, one has $(f_*\omega^{b}_{X/Y})^m\otimes O_Y(E)=f_*\omega^{mb}_{X/Y}$ for some effective divisor $E$. Since $(f_*\omega^{b}_{X/Y})^m$ and $f_*\omega^{mb}_{X/Y}$ are numerically trivial, $E$ must be zero. Hence we have $$f_*\omega^{mb}_{X/Y}=(f_*\omega^{b}_{X/Y})^m$$ for all $m\geq0$. Therefore, we have  $O_Y(L_{X/Y})=f_*\omega^b_{X/Y}$ \cite[Proposition 2.2]{FujinoMori2000}. The canonical bundle formula of $f$ has the form $$bK_{X/Y}=f^*L_{X/Y}+B,$$
where $B$ is an effective divisor satisfying $f_*O_X(mB)=O_Y$ for all $m\geq0$. Define $t_P:=\max\{t\in\Qq\mid(X,-\frac{1}{b}B+tf^*P)\text{ is sub log canonical near the generic point of P}\}$, where $P$ is a prime divisor on $Y$, and $L^{ss}_{X/Y}:=L_{X/Y}-b\sum_{P\subset Y}(1-t_P)P$ \cite[Definition 4.3]{FujinoMori2000}. Since $L^{ss}_{X/Y}$ is pseudo effective, one has $t_P=1$ for all $P$ by $L_{X/Y}\equiv0$ and \cite{BDPP}.
\end{proof}
\end{lemma}

%

\begin{theorem}\label{thm:main 2}
Let $f:X\to Y$ be a parabolic fiber space. If the $\Qq$-divisor $L^{ss}_{X/Y}$ is numerically trivial, then the integral divisor $N(\beta_{\tilde{F}})L^{ss}_{X/Y}$ is linearly equivalent to zero. 
\end{theorem}

\begin{example}
Let $f:X\to Y$ be a minimal elliptic surface, then $O_Y(12L^{ss}_{X/Y})\cong j^*O_{\Pp}(1)$, where $j:Y\to\Pp^1$ is the $j$-function. If $L^{ss}_{X/Y}\equiv0$, then $j$ is constant, hence $O_Y(12L^{ss}_{X/Y})\cong O_Y$.    
\end{example}

\begin{proof}[Proof of Theorem \ref{thm:main 2}]
We follow closely the proof of canonical bundle formula \cite[Section 5]{Mori1985classification}, \cite{FujinoMori2000}. The argument of Step4 is due to Kawamata \cite{Kawamata1982curve}. 

Step0 (Reduce to the case $\dim Y=1$). Set $N=N(\beta_{\tilde{F}})$. By \cite[Theorem 3.1]{FujinoMori2000}, $NL^{ss}_{X/Y}$ is an integral divisor, to show that $O_Y(NL^{ss}_{X/Y})\cong O_Y$, it suffices to show that this holds outside a codimension $\geq2$ closed subset of $Y$. Hence one can assume that $\dim Y=1$ by replacing $Y$ by an intersection of general hyperplane sections $H_1\cap \cdots \cap H_{\dim Y-1}$ and $X$ by $f^*H_1\cap \cdots \cap f^*H_{\dim Y-1}$.

\begin{tikzcd}
W \arrow[r, "\pi"] \arrow[rd, "g"] & X \arrow[d, "f"] \\
& Y               
\end{tikzcd}

Step1 (Covering Trick)\cite[Remark 2.6]{FujinoMori2000}. Take $0\neq\phi\in\Cc(X)$, such that $bK_X=(\phi)+D$ and $D^h$ is effective, then $f_*O_X(mD^h)=O_Y$ for all $m\geq0$. Let $W$ be a resolution of the normalization of $X$ in $\Cc(X)(\phi^{1/b})$. Then $g:W\to Y$ is a fiber space whose general fiber is a resolution of the cover branched over the unique divisor in $|bK_F|$, and $\kappa(\tilde{F})=0$, $p_g(\tilde{F})=1$, $L^{ss}_{X/Y}=bL^{ss}_{W/Y}$ \cite[Lemma 3.4]{FujinoMori2000}.

\begin{tikzcd}
W \arrow[d, "g"] & W' \arrow[l, "q"] \arrow[d, "g'"] \\
Y                & Y' \arrow[l, "p"]                
\end{tikzcd}

Step2 (Semistable reduction). Let $p:Y'\to Y$ be a finite Galois semistable reduction, and $g'$ the induced fiber space. Let $Y_0$ and $Y'_0$ be the smooth locus of $g$ and $g'$ respectively. We may assume that $g^*(Y\setminus Y_0)$ and $g'^*(Y'\setminus Y'_0)$ are simple normal crossing. Since $p^*O_Y(NL^{ss}_{W/Y})=g'_*\omega^N_{W'/Y'}=(g'_*\omega_{W'/Y'})^N$, one has $\deg g'_*\omega_{W'/Y'}=0$. Hence the Hodge metric $h'$ on $g'_*\omega_{W'/Y'}|_{Y'_0}$ is flat. By $p^*(g_*\omega_{W/Y}|_{Y_0})=g'_*\omega_{W'/Y'}|_{Y'_0}$, the Hodge metric $h$ on $g_*\omega_{W/Y}|_{Y_0}$ is flat. 

Step3. Let $H_{\Cc}=(R^dg_{0*}\Cc_{W_0})_{prim}$ be the polarized variation of Hodge structure, where $d=\dim\tilde{F}$. Then  
$g_*\omega_{W/Y}|_{Y^0}$ and its flat Hodge metric define a local subsystem of $H_{\Cc}$ of rank one, which is given by a character, say, $\chi:\pi_1(Y_0)\to\Cc^*$. By Deligne \cite[Proof of Corollary 4.2.8(iii)b]{Delige1971Hodge2}, given any $\gamma\in \pi_1(Y_0)$, $\chi(\gamma)$ is a root of unity, and $[\Qq(\chi(\gamma)):\Qq]\leq\rk H_{\Cc}=:\beta_{\tilde{F}}$. Hence if $\chi(\gamma)$ is a $k$-th root of unity, one has $\varphi(k)\leq\beta_{\tilde{F}}$. Thus $\chi^N=1$, by recalling that $N:=lcm\{k\in\Nn\mid\varphi(k)\leq\beta_{\tilde{F}}\}$.  Consequently, one has $g_*\omega^N_{W/Y}|_{Y_0}=(g_*\omega_{W/Y}|_{Y_0})^N\cong O_{Y_0}$. 

Step4. By pulling back, one also has $(g'_*\omega_{W'/Y'}|_{Y'_0})^N\cong O_{Y'_0}$. Since $\deg g'_*\omega_{W'/Y'}=0$, $g'_*\omega_{W'/Y'}|_{Y'_0}$ has unipotent hence trivial local monodromies around $Y'\setminus Y'_0$ \cite[page 69]{Kawamata1982curve}. Thus $O_{Y'}\cong(\text{The canonical extension of }g'_*\omega_{W'/Y'}|_{Y'_0})^N\cong (g'_*\omega_{W'/Y'})^N$, where the second $\cong$ is due to \cite[Lemma 1]{Kawamata1982curve}. Note that by $p^*O_Y(NL^{ss}_{W/Y})=(g'_*\omega_{W'/Y'})^N$, we have $O_Y(NL^{ss}_{W/Y})$ is torsion.

Step5. Let $0\neq s\in H^0(g_*\omega^N_{W/Y}|_{Y_0})$ be a flat section. Then $p^*s\in H^0(g'_*\omega^N_{W'/Y
}|_{Y'_0})$ is flat. Since $|p^*s|^2_{h'^N}$ grows at most logarithmically along $Y_0\setminus Y'_0$ by \cite[Theorem 6.6]{Schmid1973singularities}, one deduce that $p^*s$ extends to a global section $\widetilde{p^*s}$ of $g'_*\omega^N_{W'/Y'}$ \cite[Lemma 1]{Kawamata1982curve}. Since $\widetilde{p^*s}$ is $Gal(Y'/Y)$-invariant, it descends to a nonzero global section of $O_Y(NL^{ss}_{W/Y})$. Hence $O_Y(NL^{ss}_{W/Y})\cong O_Y$, and $O_Y(NL^{ss}_{X/Y})\cong O_Y$. 
\end{proof}


\section*{Acknowledgement}
The author thanks his advisor, Zhan Li, for helpful discussions and encouragement, and Professors Florin Ambro, Osamu Fujino, and Juanyong Wang for answering questions. A grant from SUSTech supports the author.

\bibliographystyle{alpha}

\bibliography{bibfile}

\end{document}